\documentclass[A4paper,11pt]{article}

\usepackage{graphics} 
\usepackage{float}
\usepackage{epsfig} 
\usepackage{graphicx}
\usepackage{epstopdf}
\usepackage{hyperref}

\usepackage{fix-cm}
\fontsize{23}{32}\selectfont

\usepackage{amsfonts}
\usepackage{amssymb}
\usepackage{amsmath,mathrsfs}
\usepackage{amsthm}
\usepackage{latexsym}
\usepackage{layout}
\usepackage{units}
\usepackage{relsize}
\usepackage{tikz}
\usepackage[titletoc]{appendix}
\usetikzlibrary{shapes,arrows}

\usepackage{enumerate}
\usepackage{textcomp}
\usepackage{amsbsy}
\usepackage{latexsym}
\usepackage[mathscr]{eucal}
\usepackage{amsfonts}
\usepackage{amssymb}
\usepackage{epsfig} 
\usepackage{epstopdf}
\usepackage{framed}

\usepackage[english]{babel}

\theoremstyle{plain}
\begingroup
\newtheorem{theorem}{Theorem}[section]
\newtheorem{lemma}[theorem]{Lemma}
\newtheorem{proposition}[theorem]{Proposition}

\endgroup

\theoremstyle{definition}
\begingroup
\newtheorem{definition}[theorem]{Definition}
\newtheorem{problem}{Problem}
\newtheorem{remark}[theorem]{Remark}

\endgroup

\theoremstyle{remark}
\begingroup

\endgroup

\def\O{\Omega}
\def\ds{\displaystyle}

\newcommand{\Lip}{\textup{Lip}}

\newcommand{\N}{\mathbb{N}}

\newcommand{\R}{\mathbb{R}}

\newcommand{\M}{\mathcal{M}}

\newcommand{\W}{\mathcal{W}}
\newcommand{\PP}{\mathcal{P}_1}

\DeclareMathOperator{\dive}{div}
\DeclareMathOperator{\supp}{supp}

\DeclareMathOperator{\Adm}{Adm}

\allowdisplaybreaks

\begin{document}
	
	\title{Optimal Control Problems in Transport Dynamics}

	\author{Mattia Bongini\thanks{Faculty of Mathematics, Technische Universit\"at M\"unchen, Boltzmannstrasse 3
			D-85748, Garching bei M\"unchen, Germany. {\tt mattia.bongini@ma.tum.de}}, Giuseppe Buttazzo\thanks{Dipartimento di Matematica, Universit\`a di Pisa,
			Largo B. Pontecorvo 5, Pisa, 56127, Italy.{\tt buttazzo@dm.unipi.it}}}

	\maketitle







\maketitle


\begin{abstract}
In the present paper we deal with an optimal control problem related to a model in population dynamics; more precisely, the goal is to modify the behavior of a given density of individuals via another population of agents interacting with the first. The cost functional to be minimized to determine the dynamics of the second population takes into account the desired target or configuration to be reached as well as the quantity of control agents. Several applications may fall into this framework, as for instance driving a mass of pedestrian in (or out of) a certain location; influencing the stock market by acting on a small quantity of key investors; controlling a swarm of unmanned aerial vehicles by means of few piloted drones.
\end{abstract}



\section{Introduction}\label{intro}

In recent years several models of transport dynamics have been studied; if $\rho(t,x)$ represents the density of a given population at time $t$ in a space location $x$, the evolution of $\rho$, whenever the total mass of the population is conserved, is described by means of the {\it continuity equation}
$$\frac{\partial\rho}{\partial t}(t,x)=-\dive_x\big(v(t,x)\rho(t,x)\big),$$
where $v(t,x)$ is the velocity of the population motion. The vector field $v(t,x)$ may depend on $\rho$ in a rather general way; here we are interested in the cases where
$$v(t,x)=(K*\rho)(t,x)+f(t,x),$$
being $f(t,x)$ an {\it external velocity} field, $K(t,x)$ a {\it self-interaction kernel}, and $*$ the convolution operator
$$(K*\rho)(t,x)=\int_{\Omega} K(t,x-y)\,d\rho(t,y).$$
Our ambient space is a domain $\O$ of $\R^d$, which we take bounded and regular enough; the case of unbounded domains $\O$ can be treated in a similar way with some technical modifications. Models of the kind above have been widely considered in the literature; we refer for instance to \cite{AGS,bofo13,carrillo2014swarm,pedestrian,villani} and to the references therein.

In the present paper we deal with an optimal control problem related to the dynamics above; more precisely, the goal is to modify the behavior of the density $\rho$ of the population by influencing the behavior of another population of agents interacting with $\rho$, that we denote by $\nu$. This means that the function $f$ above is of the form
$$f(t,x) = (H*\nu)(t,x) \quad \text{for every } (t,x)\in[0,T]\times\Omega,$$
for a given \emph{cross-interaction kernel} $H$. The resulting state equation governing our optimal control problem is
\begin{equation}\label{eq:introprob}
\frac{\partial\rho}{\partial t}(t,x)=-\dive_x\Big(\big((K*\rho)(t,x)+(H*\nu)(t,x)\big)\rho(t,x)\Big),
\end{equation}
with initial condition
$$\rho(0,x)=\rho_0(x)\qquad\hbox{on }\O,$$
and boundary conditions
$$v(t,x)\cdot n(x)=0\qquad\hbox{on }(0,T]\times\partial\O,$$
where
$$v(t,x)=(K*\rho)(t,x)+(H*\nu)(t,x),$$
with $K,H$ suitable convolution kernels. Notice that, by setting $f=H*\nu$, equation \eqref{eq:introprob} has the form of a continuity equation, where $f$ is an external velocity field.

The dynamics of $\nu$ is determined by the minimization of a given functional $\mathcal{J}(\nu,\rho)$ taking into account the desired behavior of $\rho$ as well as the cost of the control agents $\nu$ (whose mass is allowed to vary). It is introduced in detail in Section \ref{svariat} by using the general theory of functionals defined on the space of measures, developed in \cite{bouchitte1990new,bobu92,bobu93}. Under rather mild assumptions on $\mathcal{J}$ we establish the existence of solutions for the optimal control problem with cost functional $\mathcal{J}(\nu,\rho)$ subject to the PDE constraint \eqref{eq:introprob}.

Notice that the formulation of our control problem differs significantly, for instance, from that of \textit{mean-field games}, introduced in \cite{lasrylions}, as rather than embedding decentralized control rules inside the dynamics of $\rho$ we introduce an external control mass $\nu$ that interacts with the original population with the goal to modify its behavior.

The reason to study such infinite dimensional optimal control problems instead of their discrete counterparts lies in the so-called \textit{curse of dimensionality}, term introduced by Richard Bellman in \cite{bellmandynamicprogramming} to describe the difficulty in solving optimization problems where the dimension of the state variable (which depends on the number of agents, in this case) is large: the goal is to compute a nearly optimal control strategy that does not depend anymore on the number of agents.

Several applications may fall into our framework; for instance
\begin{itemize}
\item driving a mass of pedestrian to (or out of) a certain location using a small number of stewards;
\item trying to stabilize the stock market in order to avoid systemic failures, by acting on few key investors with a relatively limited amount of resources;
\item computing the minimal amount of manually-controlled units such that a swarm of drones performs a given task (as, for instance, wind harvesting or the recognition of a given area).
\end{itemize}

In the present paper we do not perform numerical simulations; we want to stress that this issue presents several difficulties, mainly related to the nonlocal behavior of the governing state equations and to the nonconvexity of the cost functional. Some numerical simulations of problems of similar type have been performed in \cite{albi2015invisible,albikinetic}.

After introducing the model in Section \ref{sec:preliminaries} and the class of admissible controls in Section \ref{sec:admissible}, we state in Section \ref{svariat} the optimal control problem rigorously and we study its well-posedness; some variants are also considered. Section \ref{exj} is devoted to a list of functionals falling into our framework, and Section \ref{fornasolo} to the analysis of a natural control problem arising in pedestrian dynamics.

\section{Preliminaries}\label{sec:preliminaries}

\subsection{The Wasserstein space of probability measures}

Let $\O\subset\R^d$; we denote by $\M(\O)$ the set of finite positive measures on $\O$, and by $\M_M(\O)$ the set of positive measures with total mass less than or equal to $M>0$. It is well-known that the class $\M_M(\O)$ admits a metric $d$ topologically equivalent to the weak* convergence.

The space $\mathcal{P}(\O)$ is the subset of $\M(\O)$ whose elements are the probability measures on $\O$, i.e., $\mu\in\M(\O)$ for which $\mu(\O)=1$. The space $\mathcal{P}_p(\O)$ is the subset of $\mathcal{P}(\O)$ whose elements have finite $p$-th moment, i.e.,
$$\int_\O|x|^p\,d\mu(x)<+\infty.$$
Clearly $\mathcal{P}_p(\O)=\mathcal{P}(\O)$ when $\O$ is bounded. Finally, we denote by $\mathcal{P}_c(\O)$ the subset of $\mathcal{P}(\O)$ which consists of all probability measures with compact support. 

For any $\mu\in\mathcal{P}(\R^{d_1})$ and any Borel function $f:\R^{d_1}\to\R^{d_2}$, we denote by $f_{\#}\mu\in\mathcal{P}(\R^{d_2})$ the {\it push-forward of $\mu$ through $f$}, defined by
$$f_{\#}\mu(B):=\mu(f^{-1}(B))\qquad\text{for every Borel set $B$ of }\R^{d_2}.$$
In particular, if one considers the projection operators $\pi_1$ and $\pi_2$ defined on the product space $\R^{d_1}\times\R^{d_2}$, for every $\rho\in\mathcal{P}(\R^{d_1}\times\R^{d_2})$ we call {\it first} (resp., {\it second}) {\it marginal} of $\rho$ the probability measure $\pi_{1\#}\rho$ (resp., $\pi_{2\#}\rho$). Given $\mu\in\mathcal{P}(\R^{d_1})$ and $\nu\in\mathcal{P}(\R^{d_2})$, we denote by $\Gamma(\mu,\nu)$ the subset of all probability measures in $\mathcal{P}(\R^{d_1}\times\R^{d_2})$ with first marginal $\mu$ and second marginal $\nu$.

On the set $\mathcal{P}_p(\R^d)$ we consider the Wasserstein or Monge-Kantorovich-Rubinstein distance,
\begin{equation}\label{e_Wp}
\W_p(\mu,\nu)=\inf\left\{\int_{\R^{2d}}|x-y|^p\,d\rho(x,y)\ :\ \rho\in\Gamma(\mu,\nu)\right\}^{1/p}.
\end{equation}
If $p=1$ we have the equivalent expression for the Wasserstein distance:
$$\W_1(\mu,\nu)=\sup\left\{\int_{\R^d}\varphi(x)\,d(\mu-\nu)(x)\ :\ \varphi\in\Lip(\R^d),\ \Lip_{\R^d}(\varphi)\le1\right\},$$
where $\Lip_{\R^d}(\varphi)$ stands for the Lipschitz constant of $\varphi$ on $\R^d$. We denote by $\Gamma_o(\mu,\nu)$ the set of optimal plans for which the minimum is attained, i.e.,
$$\rho\in\Gamma_o(\mu,\nu)\iff\rho\in\Gamma(\mu,\nu)\text{ and }\int_{\R^{2d}}|x-y|^p\,d\rho(x,y)=\W_p(\mu,\nu)^p.$$
It is well-known that $\Gamma_o(\mu,\nu)$ is non-empty for every $(\mu,\nu)\in\mathcal{P}_p(\R^d)\times\mathcal{P}_p(\R^d)$, hence the infimum in \eqref{e_Wp} is actually a minimum. For more details, see e.g. \cite{AGS,villani}.


\subsection{The model}

Let $T>0$ be a finite-time horizon and let $\O\subset\R^d$ be a bounded open regular set, admitting the possibility of not being convex, i.e., $\O$ may have internal ``obstacles'' and ``walls''.

The dynamics of a conserved quantity $\rho$ under the effect of an external vector field $v:[0,T]\times\R^d\to\R^d$ is described by means of the {\it continuity equation}, given by
\begin{equation}\label{continuity}
\frac{\partial\rho}{\partial t}(t,x)=-\dive_x\big(v(t,x)\rho(t,x)\big).
\end{equation}
A detailed analysis of \eqref{continuity} in the case $\rho\in\mathcal{P}(\R^d)$ can be found in \cite{AGS}. To model the interaction of $\rho$ with the possible obstacles in $\O$, we prescribe \textit{reflecting boundary conditions} of the form
$$v(t,x)\cdot n(x)=0\qquad\text{on }[0,T]\times\partial\O,$$
where $n:\partial\O\to\R^d$ is the outer normal to the boundary of $\O$.

The evolution of the measure-valued curve $\rho:[0,T]\to\mathcal{P}(\R^d)$ is then given by
\begin{equation}\label{eq:compacteq}
\begin{cases}
\ds\frac{\partial\rho}{\partial t}(t,x)=-\dive_x\big(v(t,x)\rho(t,x)\big)&\text{in }(0,T]\times\O,\\
\rho(0,x)=\rho_0(x)&\text{on }\O,\\
v(t,x)\cdot n(x)=0&\text{on }(0,T]\times\partial\O,
\end{cases}
\end{equation}
where $\rho_0$ is an initial probability distribution with support contained in the interior of $\O$.

\begin{remark}
Notice that, thanks to the boundary conditions and $\supp(\rho_0)\subseteq\mathring{\O}$, then $\supp(\rho(t))\subseteq\mathring{\O}$ for all $t \in [0,T]$.	
\end{remark}

We now proceed to clarify our notion of solution for \eqref{eq:compacteq}.

\begin{definition}\label{defsolution}
Given $\rho:[0,T]\to\mathcal{P}(\O)$ and $v:[0,T]\times\O\to\R^d$, we say that $(\rho,v)$ is a {\it solution of} \eqref{eq:compacteq} if
\begin{itemize}
\item $\rho$ is continuous with respect to the Wasserstein distance $\W_1$;
\item $\rho$ satisfies $\rho(0)=\rho_0$ and for every $\phi\in\mathcal{C}^{\infty}_0([0,T];\mathcal{C}^{\infty}_b(\O))$ it holds
$$\int^T_0\int_\O\left(\frac{\partial\phi}{\partial t}(t,x)+v(t,x)\cdot\nabla\phi(t,x)\right)\,d\rho(t,x)\,dt=0.$$
\end{itemize}
\end{definition}

Notice that no continuity assumptions are made on the velocity field $v$, the definition of solution above is given in the weak distributional sense. Our main interest lies in the case that $v$ has a specific dependency on $\rho$, namely
\begin{equation}\label{ourvel}
v(t,x):=(K*\rho)(t,x)+f(t,x),\qquad\text{for all }(t,x)\in[0,T]\times\R^d.
\end{equation}
In the expression above, the function $f:[0,T]\times\R^d\to\R^d$ is an external velocity field and $*$ denotes the convolution operator
$$(K*\rho)(t,x):=\int_{\R^d}K(t,x-y)\,d\rho(t,y).$$
Here $K:[0,T]\times\R^d\to\R^d$ is a {\it self-interaction kernel} which models the self-interaction of $\rho$. Several instances of such interaction kernels can be found in biology, chemistry and social sciences, see for instance \cite{carrillo2014swarm,CS,hegselmann2002opinion,lennard24,reynolds1987flocks,vicsek2012collective}.

\section{The class of admissible velocity fields}\label{sec:admissible}

We now turn our attention to the solutions of system \eqref{eq:compacteq}; we show that, under mild conditions on the functions $K$ and $f$ appearing in \eqref{ourvel}, they exist and are unique. The following results generalize those in \cite{FornasierSolombRossiBongini}, and are reported to keep track of the explicit dependencies of the constants.

We start by introducing the class of $\ell$-{\it admissible functions}.


\begin{definition}\label{def:adm}
Fix $T>0$ and $\ell\in L^1(0,T)$. The class $\Adm_{\ell}([0,T]\times\R^d;\R^d)$ is the set of all functions $g:[0,T]\times\R^d\to\R^d$ satisfying:
\begin{enumerate}[($i$)]
\item $g$ is a Carath\'eodory function;
\item $|g(t,x)-g(t,y)|\le\ell(t)|x-y|$ for all $(t,x),(t,y)\in[0,T]\times\R^d$;
\item $|g(t,x)|\le\ell(t)(1+|x|)$ for all $(t,x)\in[0,T]\times\R^d$.
\end{enumerate}
\end{definition}

The following result, which can be found in \cite{MFOC}, shows that $\Adm_{\ell}([0,T]\times\R^d;\R^d)$ is compact with respect to a topology interacting with the $\W_1$ convergence of measures.

\begin{theorem}\label{thm:3}
Let $\ell\in L^1(0,T)$ and $1<p<\infty$. For any $(g_n)_{n\in\N}\subset\Adm_{\ell}([0,T]\times\R^d;\R^d)$ there exists a subsequence $(g_{n_k})_{k\in\N}$ and $g\in\Adm_{\ell}([0,T]\times\R^d;\R^d)$ such that
\begin{equation}\label{locweakconv}
\lim_{k\to\infty}\int_0^T\big\langle\phi(t),g_{n_k}(t,\cdot)-g(t,\cdot)\big\rangle\,dt=0,
\end{equation}
for all $\phi\in L^{\infty}([0,T],W^{-1,p'}(\R^d,\R^d))$ such that $\supp\big(\phi(t)\big)\subset B(0,r)$ for all $t\in[0,T]$, for some $r>0$. Here the symbol $\langle\cdot,\cdot\rangle$ denotes the duality pairing between $W^{1,p}$ and its dual $W^{-1,p'}$.

Moreover, given a compact set $\O\subset\R^d$, if $(\mu_n)_{n\in\N}$ is a sequence of functions from $[0,T]$ to $\mathcal{P}(\O)$ converging to $\mu:[0,T]\to\mathcal{P}(\O)$ in the Wasserstein distance, i.e.,
$$\lim_{n\to\infty}\W_1(\mu_n(t),\mu(t))=0\qquad\mbox{for all }t\in[0,T],$$
then for all $\varphi\in \mathcal{C}_c^1(\R^d,\R^d)$ and for all $t\in[0,T]$ it holds
\begin{align}\label{compmu}
\lim_{n\to\infty}\int_0^t\big\langle\varphi,g_n(s,\cdot)\mu_n(s)\big\rangle\,ds=\int_0^t\big\langle\varphi,g(s,\cdot)\mu(s)\big\rangle\,ds.
\end{align}
In addition, the inequality
\begin{align}\label{semicont+}
\int_0^T \langle \psi(g(t,\cdot)), \mu(t) \rangle\,dt \leq \liminf_{n\to \infty}\int_0^T \langle \psi(g_n(t,\cdot)), \mu_n(t) \rangle\,dt,
\end{align}
holds for any nonnegative convex globally Lipschitz function $\psi \colon \R^d\to[0,+\infty)$.
\end{theorem}

\begin{proof}
It is straightforward to show that $\Adm_{\ell}([0,T]\times\R^d;\R^d)$ is contained within the class of Carath\'eodory functions $g:[0,T]\times\R^d\to\R^d$ satisfying
\begin{enumerate}[($a$)]
\item $g(t,\cdot)\in W^{1,\infty}_{\rm loc}(\R^d,\R^d)$ for almost every $t\in[0,T]$;
\item $|g(t,0)|\le\ell(t)$ for almost every $t\in[0,T]$;
\item $\Lip_{\R^d}(g(t,\cdot))\le\ell(t)$ for almost every $t\in[0,T]$.
\end{enumerate}
The result follows by Corollary 2.7, Theorem 2.10 and Theorem 2.12 of \cite{MFOC}.
\end{proof}
%
%

For the sake of brevity, from now on we set $\Adm_{\ell}:=\Adm_{\ell}([0,T]\times\R^d;\R^d)$. The following result, whose proof is reported in the Appendix, shows that whenever $K$ and $f$ belong to the class $\Adm_{\ell}$, a solution $\rho$ of system \eqref{eq:compacteq} exists, is unique and is uniformly continuous in time: remarkably, the modulus of continuity depends only on $\rho_0$, $T$ and $\ell$.

\begin{theorem}\label{th:bound}
	Fix $T > 0$, $\rho_0 \in \mathcal{P}_c(\R^d)$ and $\ell \in L^1(0,T)$. If $K,f \in \Adm_{\ell}$, then there exists a unique solution $\rho \in \mathcal{C}([0,T];\mathcal{P}_1(\R^d))$ of system \eqref{eq:compacteq}. Furthermore, there exist $R,L > 0$ depending only on $\rho_0$, $T$ and $\ell$ such that
	\begin{itemize}
		\item $\supp(\rho(t)) \subseteq B(0,R)$ for every $t \in [0,T]$;
		\item $\rho$ is uniformly continuous with modulus of continuity $\omega(t,s) = \ds L\int^t_s\ell(\theta)d\theta.$
	\end{itemize}
\end{theorem}

\begin{remark}\label{fromnowon}
In what follows, for the sake of simplicity, we assume that the function $\ell$ is in $L^{\infty}(0,T)$, so that Theorem \ref{th:bound} applies and $\rho$ turns out to be Lipschitz continuous with a constant $L = (2+3R)\|\ell\|_{L^{\infty}(0,T)}$ (see equation \eqref{eq:unifmodulus} in the Appendix). The more general case $\ell\in L^{1}(0,T)$ would provide $\rho\in W^{1,1}([0,T];\mathcal{P}_1(B(0,R)))$ in the sense of \cite{AGS}, and all the results below follow along the same lines. This assumption helps us to keep the notation compact without any loss of generality.
\end{remark}

\section{The variational problem}\label{svariat}


We now pass to study how to control the behavior of $\rho$ by means of another mass of individuals $\nu$ -- representing, for instance, the officers and stewards of a building to be evacuated -- whose evolution is obtained by the minimization of a suitable given cost functional $\mathcal{J}$. Formally, this means coupling the dynamics of $\rho$ with $\nu$ through an interaction kernel $H \in \Adm_{\ell}$ as follows
\begin{equation}\label{eq:newconteq}
\begin{cases}
\ds\frac{\partial\rho}{\partial t}(t)=-\dive_x\Big(\big((K*\rho)(t)+(H*\nu)(t)\big)\rho(t)\Big)&\text{ for }t\in(0,T],\\
\rho(0)=\rho_0.&
\end{cases}
\end{equation}
Notice that system \eqref{eq:newconteq} is again of the form of system \eqref{eq:compacteq} with velocity field
$$v(t,x):=(K*\rho)(t,x)+(H*\nu)(t,x),$$
hence of the same nature of \eqref{ourvel}. This implies, by Theorem \ref{th:bound} and Remark \ref{fromnowon}, that the solutions of \eqref{eq:newconteq} are Lipschitz curves with a Lipschitz constant $L$ and values in $\mathcal{P}_1(\R^d)$, whose support is uniformly bounded in time inside $\O$, i.e., they belong to the class
$$\Lip_L([0,T];\mathcal{P}_1(\O))=\big\{\rho\in\mathcal{C}([0,T];\mathcal{P}_1(\O))\;:\;\W_1(\rho(t),\rho(s))\le L|t-s|\ \forall t,s\in[0,T]\big\}.$$

We now make the assumption that, similarly to $\rho$, also the control mass $\nu$ has a characteristic limit speed (or acceleration) $L'>0$. We thus prescribe $\nu$ to belong to the class
\begin{align*}
\Lip_{L'}([0,T];\M_M(\O))=\big\{\nu\in\mathcal{C}([0,T];&\M_M(\O))\;:\\
&d(\nu(t),\nu(s))\le L'|t-s|\ \forall t,s\in[0,T]\big\},
\end{align*}
where $d$ is a metric on $\M_M(\O)$ equivalent to the weak* topology. The modeling reason for considering $\M_M(\O)$ as the space where the curve $\nu$ takes its values is that we want to allow the mass of $\nu$ to change over the time, up to a maximal mass $M$ (in the interpretation of $\nu$ as the probability distribution of stewards, we would like to change their number as the needs come).

The dynamics of $\nu$ is given by the minimization of a cost functional $\mathcal{J}$, encoding a certain goal that $\rho$ and $\nu$ have to reach, subject to system \eqref{eq:newconteq}, which prescribes the evolution of $\rho$. We assume that the cost functional
$$\mathcal{J}:\Lip_{L'}([0,T];\M_M(\O))\times\Lip_L([0,T];\mathcal{P}_1(\O))\to\R\cup\{+\infty\},$$
that we optimize in our control problem, satisfies the following assumptions:
\begin{description}
\item[(J1)] $\mathcal{J}$ is bounded from below;
\item[(J2)] $\mathcal{J}$ is lower semicontinuous with respect to the pointwise (in time) weak* convergence of measures, i.e., for any $(\nu_n,\rho_n)_{n\in\N}\subset \Lip_{L'}([0,T];\M_M(\O))\times\Lip_L([0,T];\mathcal{P}_1(\O))$ such that $(\nu_n(t),\rho_n(t))\to(\nu(t),\rho(t))$ weakly* for every $t\in[0,T]$, it holds
$$\mathcal{J}(\nu,\rho)\le\liminf_{n\to\infty}\mathcal{J}(\nu_n,\rho_n).$$
\end{description}

Some examples of interesting cost functionals $\mathcal{J}$ satisfying (J1) and (J2) are listed in Section \ref{exj}. We can now state the optimal control problem we study henceforth.
\textbf{
\begin{problem}\label{prob1}
\textnormal{
Given $\rho_0\in\mathcal{P}_c(\R^d)$ and $K,H\in\Adm_{\ell}$, solve
$$\min\Big\{\mathcal{J}(\nu,\rho)\;:\;(\nu,\rho)\in\Lip_{L'}([0,T];\M_M(\O))\times\Lip_L([0,T];\mathcal{P}_1(\O))\Big\}$$
subject to the state equation \eqref{eq:newconteq}.}
\end{problem}}

It is straightforward to see that Problem \ref{prob1} can then be rewritten as
\begin{equation}\label{trueprob1}
\begin{split}
\min\Big\{\mathcal{J}(\nu,\rho)&+\chi_A(\nu,\rho)\;:\\
&\;(\nu,\rho)\in\Lip_{L'}([0,T];\M_M(\O))\times\Lip_L([0,T];\mathcal{P}_1(\O))\Big\}
\end{split}
\end{equation}
where $\chi_A$ is the characteristic function (with value $0$ on $A$ and $+\infty$ elsewhere) of the set
\[\begin{split}
A=\Big\{(\nu,\rho):[0,T]\to&\M_M(\O)\times\mathcal{P}_1(\O)\;:\;\rho(0)=\rho_0\text{ and}\\
&\ds\frac{\partial\rho}{\partial t}(t)=-\dive_x\Big(\big((K*\rho)(t)+(H*\nu)(t)\big)\rho(t)\Big)\text{ for }t\in (0,T]\Big\}.
\end{split}\]

\begin{lemma}\label{le:eqclosed}
The set $A$ is closed under the topology of pointwise weak* convergence of measures. Therefore, $\chi_A:\Lip_{L'}([0,T];\M_M(\O))\times\Lip_{L}([0,T];\mathcal{P}_1(\O))\to\R\cup\{+\infty\}$ also satisfies the assumption (J2).
\end{lemma}

\begin{proof}
Take $(\nu_n,\rho_n)_{n\in\N}\subset A$ such that $(\nu_n(t),\rho_n(t))\to(\nu(t),\rho(t))$ weakly* for every $t\in[0,T]$. By Definition \ref{defsolution}, to prove that $(\nu,\rho)\in A$ we have to show that $\rho(0)=\rho_0$ and for every $\phi\in\mathcal{C}^\infty_0([0,T];\mathcal{C}^\infty_b(\O))$
$$\int^T_0\int_\O\Big(\frac{\partial\phi}{\partial t}(t,x)+\big((K*\rho)(t,x)+(H*\nu)(t,x)\big)\cdot\nabla\phi(t,x)\Big)\,d\rho(t,x)\,dt=0.$$
The fact that $\rho(0)=\rho_0$ simply follows from the assumption that $\rho_n(0)=\rho_0$ for every $n\in\N$ and the uniqueness of the weak* limit. Since $(\nu_n,\rho_n)\in A$ we have that for every $\phi\in\mathcal{C}^\infty_0([0,T];\mathcal{C}^\infty_b(\O))$
$$\int^T_0\int_\O\Big(\frac{\partial\phi}{\partial t}(t,x)+\big((K*\rho_n)(t,x)+(H*\nu_n)(t,x)\big)\cdot\nabla\phi(t,x)\Big)\,d\rho_n(t,x)\,dt=0.$$
Hence, by the weak* convergence, the regularity of the test functions and the dominated convergence theorem, we obtain
$$\lim_{n\to\infty}\int^T_0\int_\O\frac{\partial\phi}{\partial t}(t,x)\,d\rho_n(t,x)\,dt=\int^T_0\int_\O\frac{\partial\phi}{\partial t}(t,x)\,d\rho(t,x)\,dt.$$
For the same reasons, and the continuity of $H(t,\cdot)$, we have
$$\lim_{n\to\infty}\nabla\phi(t,x)\cdot\int_\O H(t,x-y)\,d\nu_n(t,y)=\nabla\phi(t,x)\cdot\int_\O H(t,x-y)\,d\nu(t,y)$$
for every $(t,x)\in[0,T]\times\O$, while its admissibility and the uniform compact support of the measures gives us the upper bound
\begin{equation}
\label{eq:boundinL1}
\left|\nabla\phi(t,x)\cdot\int_\O H(t,x-y)\,d\nu_n(t,y)\right|\le M\ell(t)\big(1+\delta(\O)\big)\sup_{(t,x)\in[0,T]\times\O}|\nabla\phi(t,x)|,
\end{equation}
where we have set
\begin{equation}
\label{eq:deltaB}
\delta(\Omega)=\sup\big\{|x|\ :\ x\in \O\big\}.
\end{equation}
Notice that the bound \eqref{eq:boundinL1} belongs to $L^1(0,T)$ (notice that this also holds if $\ell$ simply belongs to $L^1(0,T)$), and that the same holds true with $K$ in place of $H$, $\rho_n$ in place of $\nu_n$ and $\rho$ in place of $\nu$. By the dominated convergence theorem and the compact support of the measures, we obtain finally
\[\begin{split}
\lim_{n\to\infty}\int^T_0\int_\O&\big((K*\rho_n)(t,x)+(H*\nu_n)(t,x)\big)\cdot\nabla \phi(t,x)\,d\rho_n(t,x)\,dt\\
&=\int^T_0\int_\O\big((K*\rho)(t,x)+(H*\nu)(t,x)\big)\cdot\nabla\phi(t,x)\,d\rho(t,x) \,dt,
\end{split}\]
which concludes the proof.
\end{proof}

The compactness of the set $\Lip_L([0,T];\mathcal{P}_1(\O))$, where $\mathcal{P}_1(\O)$ is endowed with the $\W_1$ metric, was already discussed in the proof of Theorem \ref{th:bound}. The following result shows that also the set $\Lip_{L'}([0,T];\M_M(\O))$, where $\M_M(\O)$ is equipped with the metric of the weak* convergence, is compact.

\begin{lemma}\label{compactM}
Consider $\M_M(\O)$ equipped with the metric of weak* convergence. Then, the set $\Lip_{L'}([0,T];\M_M(\O))$ is compact with respect to the uniform convergence.
\end{lemma}

\begin{proof}
Without loss of generality, assume $M=1$. Notice that for a positive measure $\mu\in\M(\O)$ its total variation $|\mu|$ coincides with $\mu(\O)$ itself; hence, the set $\M_1(\O)$ coincides with the closed unit ball $\{\mu\in\M(\O):\mu(\O)\le1\}$, which is compact in the weak* topology from the Banach-Alaoglu Theorem. Therefore, consider a sequence $(\nu_n)_{n\in\N}\subset\Lip_{L'}([0,T];\M_M(\O))$. Similarly to the proof of Theorem \ref{th:bound}, we have that
\begin{itemize}
\item $(\nu_n)_{n\in\N}$ is equicontinuous and is contained in a closed subset of the set $\mathcal{C}([0,T];\M_M(\O))$, because of the uniform bound on the Lipschitz constant;
\item for every $t\in[0,T]$, the sequence $(\nu_n(t))_{n\in\N}$ is relatively compact in $\M_M(\O)$ equipped with the weak* topology, since this metric space is compact.
\end{itemize}
Hence, an application of the Ascoli-Arzel\`a Theorem for functions with values in a metric space concludes the proof.
\end{proof}

We are now ready to address the well-posedness of Problem \ref{prob1}.

\begin{theorem}\label{th:mainmin}
Problem \ref{prob1} admits a solution.
\end{theorem}
\begin{proof}
We prove the statement by means of the direct methods in the Calculus of Variations. Rewrite Problem \ref{prob1} in the form \eqref{trueprob1} and notice that from the hypothesis (J1) the functional $\mathcal{J}+\chi_A$ is bounded from below. We can thus consider a minimizing sequence $(\nu_n,\rho_n)\in\Lip_{L'}([0,T];\M_M(\O))\times\Lip_L([0,T];\mathcal{P}_1(\O))$, which by Lemma \ref{compactM} admits a subsequence uniformly (and thus pointwise) converging to some $(\nu,\rho)\in\Lip_L([0,T];\M_M(\O))\times\Lip_{L'}([0,T];\mathcal{P}_1(\O))$. By Lemma \ref{le:eqclosed} and hypothesis (J2), the functional $\mathcal{J}+\chi_A$ is lower semicontinuous with respect to the pointwise weak* convergence, and this concludes the proof.
\end{proof}

For  use  below,  in  particular  in  Section  \ref{fornasolo},  we  mention  the  following remark.

\begin{remark}\label{rem:extraterm}
In several applications, the cost functional $\mathcal{J}$ as well as the PDE constraint \eqref{eq:newconteq} may depend on some extra term $f\in\mathcal{X}$, where $\mathcal{X}$ is a function space with topology $\tau$. Whenever it is possible to rewrite the problem as
$$\min\Big\{\mathcal{F}(\nu,\rho,f)\;:\;(\nu,\rho,f)\in\Lip_{L'}([0,T];\M_M(\O))\times\Lip_L([0,T];\mathcal{P}_1(\O))\times\mathcal{X}\Big\}$$
for a certain cost functional $\mathcal{F}$, then Theorem \ref{th:mainmin} is still valid provided that
\begin{itemize}
\item $\mathcal{X}$ is compact with respect to the topology $\tau$;
\item for every $(\rho,\nu)$, the functional $\mathcal{F}(\rho,\nu,\cdot)$ is lower semicontinuous with respect to the topology $\tau$. Note that in this formulation the state equation is included in the functional $\mathcal{F}$ as done in \eqref{trueprob1}.
\end{itemize}
\end{remark}

\section{The cost functional $\mathcal{J}$}\label{exj}

In this section we show some examples of cost functionals $\mathcal{J}$ appearing in Problem \ref{prob1}. Clearly, any linear combination of the following terms is still a valid functional for which Theorem \ref{th:mainmin} applies. We start with a preliminary result on lower semicontinuous functionals defined on $\M_M(\O)$ (see for instance \cite{bu89} or Lemma 1.6 of \cite{santambrogiooptimal}).

\begin{proposition}\label{prop:intlsc}
Let $X$ be a metric space and let $f:X\to\R\cup\{+\infty\}$ be a lower semicontinuous function bounded from below. Then the functional $\mathcal{J}:\M_M(X)\to\R\cup\{+\infty\}$ defined by
$$\mathcal{J}(\mu)=\int_X f(x)\,d\mu(x)\qquad\text{for every }\mu\in\M_M(X)$$
is lower semicontinuous with respect to the weak* convergence of measures.
\end{proposition}

In several applications of optimal control in opinion dynamics and crowd motion (see \cite{FornasierSolombRossiBongini,fornasier2014mean}), the functional to be minimized consists of a {\it Lagrangian} term of the form
$$\mathcal{J}_1(\nu,\rho)=\int^T_0 L\big(\nu(t),\rho(t)\big)\,dt.$$
The Lagrangian $L:\M_M(\O)\times\mathcal{P}_1(\O)\to\R\cup\{+\infty\}$ may prescribe, for instance, a certain mutual interaction between the measures $\nu$ and $\rho$, or one can use $L$ to model the distance to the basin of attraction of the target configurations of the measure $\rho$, as in \cite{caponigro2015sparse}. In this case, in order for $\mathcal{J}_1$ to be lower semicontinuous with respect to the pointwise weak* convergence, the lower semicontinuity of $L$ with respect to the weak* convergence suffices. Some interesting particular cases of the functional $\mathcal{J}_1(\nu,\rho)$ above are listed below.

\begin{enumerate}
\item Our decision to let the mass of $\nu$ vary comes from the choice to allow the optimization of the quantity of control agents, in accordance with the goal to achieve.
 We can model the cost of employing a quantity $\nu$ of agents at time $t$ by considering the Lagrangian
$$L\big(\nu(t),\rho(t)\big)=\int_\O f(t,x)\,d\nu(t,x).$$
Here $f:[0,T]\times\O\to[0,+\infty]$ is a lower semicontinuous function, for instance
$$f(t,x)=c(t)|x-x_0|^p,$$
where $p\ge0$, $c(t)$ is a nonnegative integrable function, and $x_0$ represents a sort of \textit{manpower storage room}.

The addition of the term $\mathcal{J}_1$ with the above choice of $L$ to the general cost functional $\mathcal{J}$ can be used to penalize the mass of $\nu$.

\item A common example is the one where we require the dynamics of the measure $\rho$ to satisfy a specific feature, like the collapse of one of its moments or marginals. An example is given by alignment models like the Cucker-Smale one (see \cite{caponigro2015sparse}), where one is interested in a population of individuals $\rho$ function of a spatial variable $x\in\R^d$ and a consensus or velocity variable $v \in \R^d$. The goal of a control strategy, in this case, is to force the alignment of the group, which in term of the state variables means that all the velocities $v$'s tend to coincide. This is ensured by minimizing at every instant $t$ and for any individual with velocity $v$ the square distance between $v$ and the mean $\overline{v}(t)=\int_{\R^{2d}}w\,d\rho(t,x,w)$:
$$\int_{\R^{2d}}\big|v-\overline{v}(t)\big|^2\,d\rho(t,x,v).$$
In more general terms, given a projection $\pi:\R^d\to\R^k$ (with $k\le d)$, then the Lagrangian $L$ may be of the form
$$L\big(\nu(t),\rho(t)\big)=\int_\O\left|x-\int_\O y\,d\pi_{\#}\rho(t,y)\right|^2\,d\pi_{\#}\rho(t,x).$$
Indeed, denoting by $\overline{x}(t)=\int_\O y\,d\pi_{\#}\rho(t,y)$ the center of mass of the projection of $\rho$ at time $t$, the minimization of the above Lagrangian leads to the convergence of $\pi_{\#}\rho(t)$ to the measure $m(t)=\delta_{\overline{x}(t)}.$ For a similar problem in the context of the Hegselmann-Krause model for opinion formation, see \cite{borzikrause}.


\item Another relevant particular case of the functional $\mathcal{J}_1$ is given by
\begin{equation}\label{eq:evacuationfun}
\mathcal{J}_2(\nu,\rho)=\int^T_0\int_C d\rho(t,x)dt
\end{equation}
where $C$ is a given subset of $\O$. By Proposition \ref{prop:intlsc}, the functional $\mathcal{J}_2$ above is pointwise weakly* lower semicontinuous as soon as $C$ is an open set. This also happens when $C$ is closed (which is the most common case in optimal evacuation problems), with $\partial C$ Lebesgue negligible, and $\rho_0$ is in $L^1(\O)$. Minimizing this functional corresponds to the evacuation of $\rho$ from the set $C$.
\item In the cases where a desired final configuration $\overline{\rho}$ of $\rho$ is given, we may use one of the following functionals
$$\mathcal{J}_3(\nu,\rho)=\int^T_0\W_1(\rho(t),\overline{\rho})\,dt\qquad\text{or}\qquad\mathcal{J}_4(\nu,\rho)=\W_1(\rho(T),\overline{\rho})$$
to force $\rho(T)$ to adhere to $\overline{\rho}$. As already noticed, the $\W_1$ distance is continuous with respect to the weak* convergence whenever the measures have uniformly compact support.

There is a slight difference between the two functionals above. The first one prescribes a somewhat \textit{greedy} approach for the optimization procedure, by asking that the distance $\W_1(\rho(t),\overline{\rho})$ cannot be too large on average (in time). The second one, instead, allows a greater freedom in the behavior of $\nu$ which is only prescribed to have a final distribution $\rho(T)$ as close as possible to $\overline\rho$. However, one should always be aware that a greater liberty may translate into a more difficult numerical implementation (since the number of degrees of freedom may grow out of control), which is an ingredient that should play a relevant role in the design of any control problem.

We remark here the connection with the Benamou-Brenier formulation of transport problems \cite{benbre00} where the initial and the final configurations $\rho(0)$ and $\rho(T)$ are both prescribed and the kinetic energy of the system has to be minimized.

\item The adoption of the space $\Lip_{L'}([0,T];\M_M(\O))$ is made for the sake of generality. However, it is often the case that the dynamics of $\nu$ possesses some extra structure which let $\nu$ belong to a narrower subset of $\Lip_{L'}([0,T];\M_M(\O))$. For instance, when $\nu$ represents a conserved quantity in time, we already noticed that its evolution can be described by means of a continuity equation like
$$\frac{\partial\nu}{\partial t}(t,x)=-\dive_x\Big(\big(v(t,x)+u(t,x)\big)\nu(t,x)\Big),$$
where only the component $u$ of the external velocity field is optimized (specifically, $v$ could be the drift depending on the interaction with the other agents, like in \eqref{ourvel}, hence $u$ stands for the optimal strategy {\it subject to} the underlying dynamics). We come back to this case in Section \ref{fornasolo}. Several families of PDEs (Fokker-Planck, Vlasov, etc$\ldots$) determine subsets $B\subseteq \Lip_{L'}([0,T];\M_M(\O))$ which are closed under the pointwise weak* convergence of measures: in all those cases, the functional
$$\mathcal{J}_5(\nu,\rho)=\chi_B(\nu,\rho)$$
is lower semicontinuous with respect to this topology, and can be used in the context of Problem \ref{prob1}.


\item Problem \ref{prob1} does not prescribe any constraint on the dynamics of the control agents $\nu$, except for its maximal speed. One way to impose extra conditions on the curve $\nu$ can be by means of the restriction to a closed subset of $\Lip_{L'}([0,T];\M_M(\O))$, like the set of solutions of a particular PDE, as argued above. This is a very powerful tool from the modeling point of view, but it may create extra difficulties if we want to identify the minimizers of $\mathcal{J}$, unless optimality conditions are available, as for instance in \cite{FornasierSolombRossiBongini}. Another way to give more structure to the dynamics of $\nu$ is to embed the desired features of it inside the functional $\mathcal{J}$, like
$$\mathcal{J}_6(\nu,\rho)=\int^T_0\int_{\O^2}Q(x,y)\,d\nu(t,x)\,d\nu(t,y)\,dt,$$
where $Q:\R^d\times\R^d\to\R$ is a lower semicontinuous function. The functional $\mathcal{J}_6$ is clearly lower semicontinuous with respect to pointwise weak* convergence by Proposition \ref{prop:intlsc}, and forces $\nu$ to self-interact via the kernel $Q$. For instance, if we want to avoid high concentrations of $\nu$, we could opt for kernels like
$$Q(x,y)=-|x-y|^p,\qquad\text{or}\qquad Q(x,y)=|x-y|^{-p},$$
while if, on the contrary, we want $\nu$ to remain as concentrated as possible, we may consider
$$Q(x,y) = |x-y|^p.$$

\item Another interesting class of functionals to model the cost of $\nu$ is given by
$$\int_0^T\left[\int_\O h\big(t,\nu^a(t,x)\big)\,dx+\sum_{x\in\O}k\big(\nu^\#(t,x)\big)\right]\,dt\;,$$
where $\nu^a$ and $\nu^\#$ are respectively the absolutely continuous and atomic parts of $\nu$ (hence, the sum over all $x\in\O$ reduces to the atoms $x$ of $\nu$), while the function $h$ (resp. $k$) is nonnegative and convex (resp. concave), with the properties $h(0)=k(0)=0$ and such that the slopes of $h$ at infinity and of $k$ at $0$ are $+\infty$. This class of functionals over the measures has been studied in \cite{bouchitte1990new}, where it is proved their weak* lower semicontinuity. To this class belongs for instance the Mumford-Shah functional (see \cite{mumfordshah}), that can be obtained by taking
$$h(s)=s^2,\qquad k(s)=1_{\R\setminus\{0\}}(s)=
\begin{cases}
0&\hbox{if }s=0\\ 1&\hbox{otherwise.}
\end{cases}$$
Another interesting choice of the functions $h$ and $k$ is
$$h(s)=\chi_{\{0\}}(s)=
\begin{cases}
0&\hbox{if }s=0\\ +\infty&\hbox{otherwise,}
\end{cases}
\qquad
k(s)=1_{\R\setminus\{0\}}(s),$$
which gives
$$\mathcal{J}_7(\nu,\rho)=\int^T_0\mathcal{H}^0\big(\nu(t)\big)\,dt,$$
being $\mathcal{H}^0$ the counting measure. This functional has a drastic effect: it forces the measure $\nu$ to be discrete at every instant, i.e.,
$$\nu(t)=\sum^N_{i=1}\delta_{x_i(t)},\qquad\text{for some }x_i(t)\in\O,$$
and minimizes the number $N$ of atoms it is supported on. However, the points $x_i(t)$ cannot vary ``too wildly'' on $\O$ in time since $\nu\in\Lip_{L'}([0,T];\M_M(\O))$: indeed, belonging to this class implies that, for any $t,s\in[0,T]$, the atoms $x_i(t)$ of $\nu(t,\cdot)$ and $y_j(s)$ of $\nu(s,\cdot)$ must satisfy
$$d\left(\sum^N_{i=1}\delta_{x_i(t)},\sum^{N'}_{j=1}\delta_{y_j(s)}\right)\le L'|t-s|.$$
\end{enumerate}

\section{A control problem in pedestrian dynamics} \label{fornasolo}


In this section, we consider a rather natural control problem arising in pedestrian dynamics, that can be reformulated as Problem \ref{prob1} with a specific choice of the cost functional $\mathcal{J}$.
\textbf{
\begin{problem}\label{prob2}
\textnormal{Given $\rho_0,\nu_0\in\mathcal{P}_c(\O)$, $C\subset\O$, and $K_1,K_2,H_1,H_2\in\Adm_\ell$, solve
$$\min_{u\in\Adm_\ell}\int^T_0\left(\int_C d\rho(t,x)+\int_\O|u(t,x)|^p\,dx\right)\,dt$$
subject to
\begin{equation}\label{eq:conteq}
\begin{cases}
\ds\frac{\partial\rho}{\partial t}(t)=-\dive_x\Big(\big((K_1*\rho)(t)+(H_1*\nu)(t)\big)\rho(t)\Big)&\text{for }t\in(0,T],\\
\ds\frac{\partial\nu}{\partial t}(t)=-\dive_x\Big(\big((K_2*\rho)(t)+(H_2*\nu)(t)+u(t)\big)\nu(t)\Big)&\text{for }t\in(0,T],\\
\rho(0)=\rho_0,&\\
\nu(0)=\nu_0.&
\end{cases}
\end{equation}}
\end{problem}}

In Problem \ref{prob2}, the two populations $\rho$ and $\nu$ are interacting via the kernels $K_1,K_2,H_1,H_2$. In addition, the population $\nu$ is trying to optimize its trajectory (the function $u$) in order to reach the goal encoded in the cost functional. As already discussed in the previous section, the goal is to evacuate the measure $\rho$ from the set $C$, while at the same time penalizing too high values of the optimized velocity field $u$. This problem has been treated extensively, especially under the further constraint for $\nu$ to be atomic, in \cite{albi2015invisible,bepito12,fornasier2014mean,MFOC}. Notice that, since the control mass $\nu$ is subjected to a continuity equation, its mass remains constant, and hence we may assume $\nu\in\mathcal{P}_1(\O)$, instead of the more general $\nu\in\M_M(\O)$.

\begin{figure}[ht!]
\centering
\includegraphics[width=0.32\textwidth]{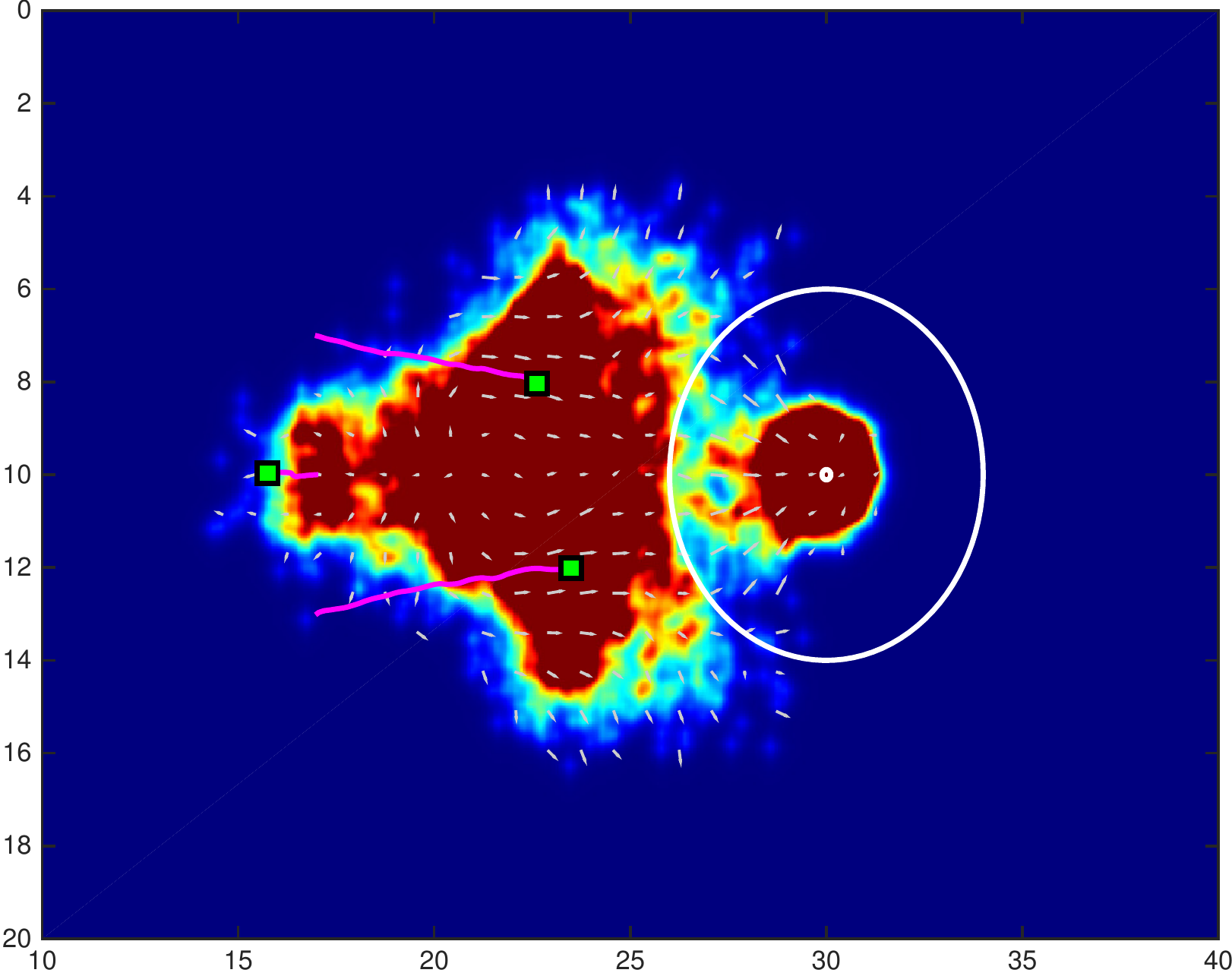}
\includegraphics[width=0.32\textwidth]{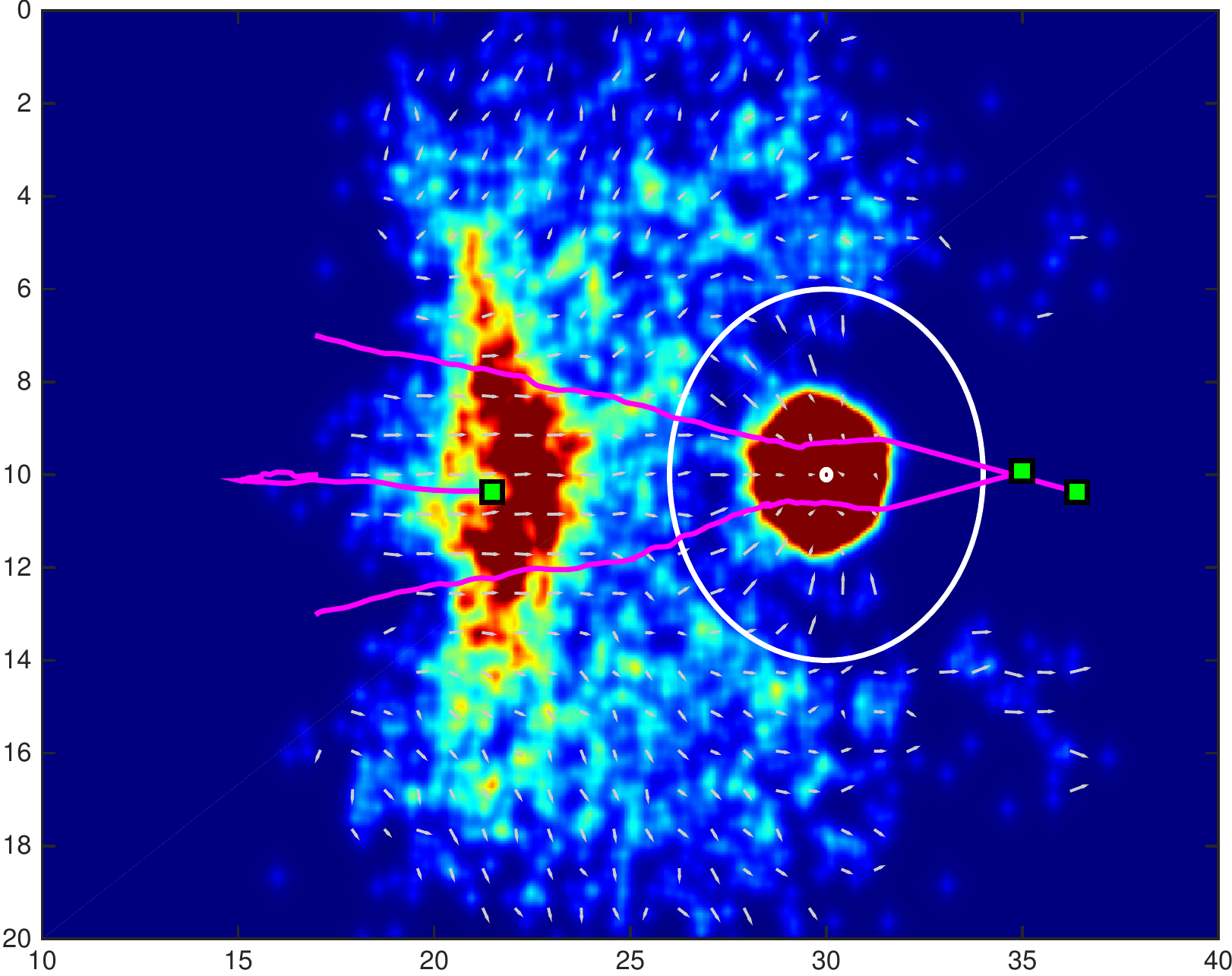}
\includegraphics[width=0.32\textwidth]{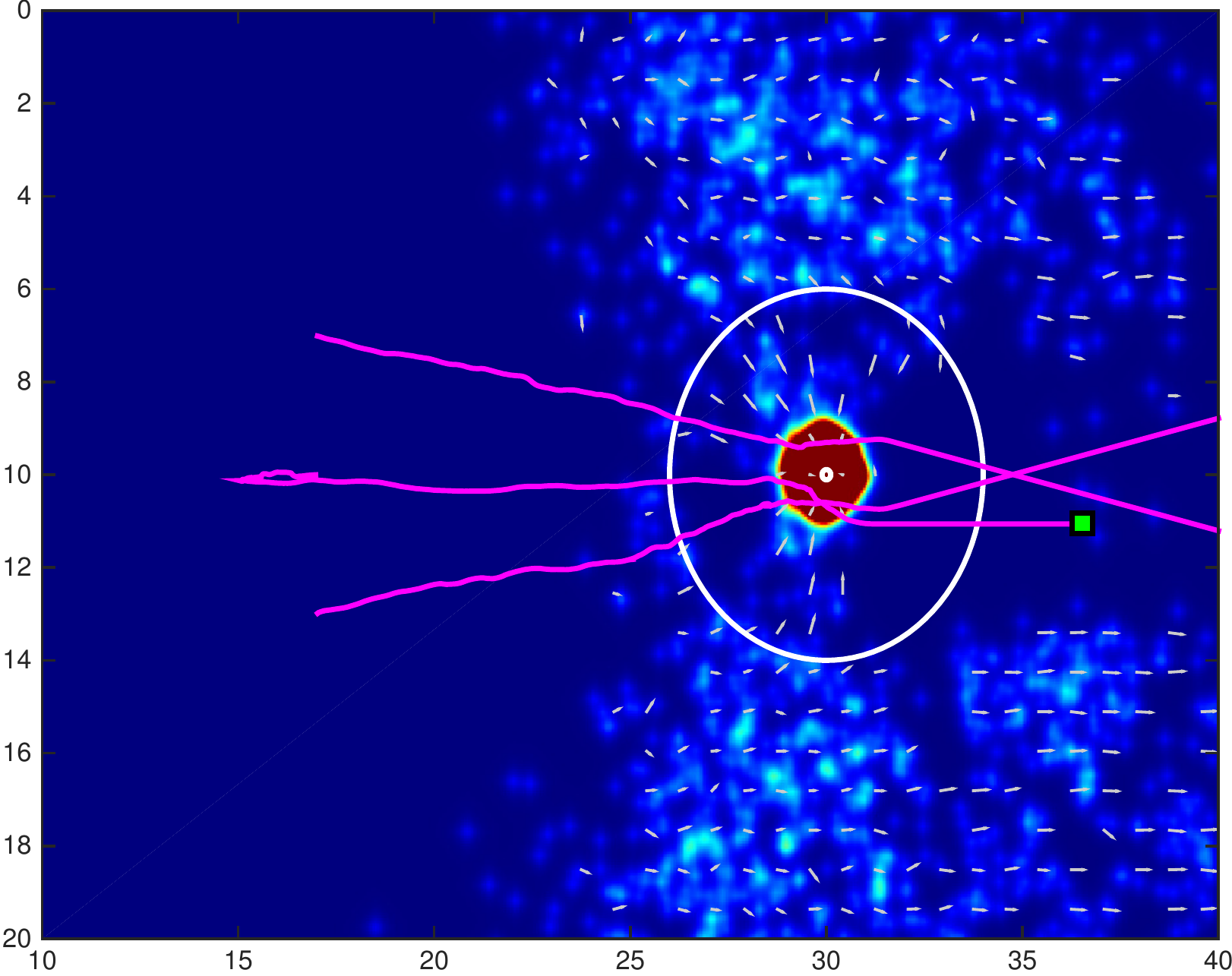}
\caption{An instance of Problem \ref{prob2} where $\nu$ is the empirical measure centered on 3 atoms. The area to be evacuated is $\R^2$ except one point, representing the exit. (Courtesy of the authors of \protect\cite{albi2015invisible}).}
\label{fig:meso}
\end{figure}

To help the reader visualize the setting of Problem \ref{prob2}, in Figure \ref{fig:meso} we report the control strategy adopted by a measure $\nu$ with three atoms (in green) to let a continuous mass $\rho$ (in red) evacuate the area via the exit located at the center of the circle. The exit is only visible to the agents $\rho$ inside the circle, while $\nu$ knows the entire environment, instead. In this situation, the interaction kernels are all repulsive at short-range, since $\rho$ and $\nu$ model pedestrians which cannot overlap in space. Therefore, if all the mass $\rho$ accumulates around the exit, a big queue would be formed due to self-repulsion. The mass $\nu$ avoids this by letting one of its atom wait before helping the portion of $\rho$ surrounding him: only after part of $\rho$ is already evacuated this atom moves and leads its portion of $\rho$ to the exit. In this way the congestion is much lower and the evacuation faster.

In order to establish the well-posedness of Problem \ref{prob2}, we rewrite it as
\[\begin{split}
\min\Big\{&\mathcal{J}(\nu,\rho,u)+\chi_A(\nu,\rho,u)\ :\\
&(\nu,\rho,u)\in\Lip_{L'}([0,T];\mathcal{P}_1(\O))\times\Lip_L([0,T];\mathcal{P}_1(\O))\times\Adm_\ell\Big\}
\end{split}\]
where
$$\mathcal{J}(\nu,\rho,u)=\int^T_0\left(\int_C d\rho(t,x)+\int_\O|u(t,x)|^p\,dx \right)dt,$$
and
\begin{align*}
A=\big\{(\nu,\rho,u)\in\Lip_{L'}([0,T];\mathcal{P}_1(\O))\times\Lip_L([0,T];\mathcal{P}_1&(\O))\times\Adm_\ell\;:\\
&\;(\nu,\rho,u)\text{ satisfies }\eqref{eq:conteq} \big\}.
\end{align*}
Notice that $\Adm_\ell$ is compact with respect to convergence in the sense of \eqref{locweakconv} (see Theorem \ref{thm:3}). Therefore, if we define $\tau$ to be the topology on the product space
$$Y:=\Lip_{L'}([0,T];\mathcal{P}_1(\O))\times\Lip_L([0,T];\mathcal{P}_1(\O))\times\Adm_\ell$$
associated with the pointwise weak* topology on $(\rho,\nu)$ and with the convergence \eqref{locweakconv} on $u$, then, by Theorem \ref{thm:3} and Lemma \ref{compactM} it follows that $\langle Y,\tau\rangle$ is a compact topological space. By Remark \ref{rem:extraterm}, the next two Lemmas are sufficient to conclude that Problem \ref{prob2} admits a solution.

\begin{lemma}
The set $A$ is a closed subset of $Y$ with respect to the topology $\tau$.
\end{lemma}

\begin{proof}
Take $(\nu_n,\rho_n,u_n)_{n\in\N}\subset A$ such that $(\nu_n(t),\rho_n(t))\to(\nu(t),\rho(t))$ weakly* for every $t\in[0,T]$ and $u_n\to u$ in the sense of \eqref{locweakconv}. Very much likely as in the proof of Lemma \ref{le:eqclosed}, we need to show that for every $\phi\in\mathcal{C}^\infty_0([0,T];\mathcal{C}^\infty_b(\O))$ it holds $\rho(0)=\rho_0$ and
$$\int^T_0\int_\O\left(\frac{\partial\phi}{\partial t}(t,x)+\Big((K_1*\rho)(t,x)+(H_1*\nu)(t,x)\Big)\cdot\nabla\phi(t,x)\right) d\rho(t,x)\,dt=0,$$
as well as $\nu(0)=\nu_0$ and
\begin{equation*}\int^T_0\!\!\int_\O\!\left(\frac{\partial\phi}{\partial t}(t,x)\!+\!\Big((K_2*\rho)(t,x)\!+\!(H_2*\nu)(t,x)\!+\!u(t,x)\Big)\!\cdot\!\nabla\phi(t,x)\!\right)\!d\nu(t,x)dt=0.
\end{equation*}
Using the same argument as in the proof of Lemma \ref{le:eqclosed}, we can prove that the integrals above are limit as $n\to\infty$ of the same integrals with $\nu_n,\rho_n,u_n$ in place of $\nu,\rho,u$, the only exception being the limit
$$\lim_{n\to\infty}\int^T_0\int_\O u_n(t,x)\cdot\nabla\phi(t,x)\,d\nu_n(t,x)\,dt=\int^T_0\int_\O u(t,x)\cdot\nabla\phi(t,x)\,d\nu(t,x)\,dt.$$
However, this limit is a consequence of \eqref{compmu}, since the sequence $(u_n)_{n\in\N}$ belongs to $\Adm_\ell$.
\end{proof}

\begin{lemma}
The functional $\mathcal{J}$ is lower semicontinuous with respect to the product topology $\tau$.
\end{lemma}

\begin{proof}
From \eqref{eq:evacuationfun}, we know that the functional $\mathcal{J}(\cdot,\cdot,u)$ is semicontinuous in the space $\Lip_{L'}([0,T];\mathcal{P}_1(\O))\times\Lip_L([0,T];\mathcal{P}_1(\O))$ for all $u\in\Adm_\ell$. Hence, if we show that the term
\begin{align}\label{eq:normp}
\int^T_0\int_\O|u(t,x)|^p\,dx\,dt,
\end{align}
is lower semicontinuous in $\Adm_\ell$ with respect to the convergence \eqref{locweakconv}, the statement is proved. However, it suffices to observe that, by Theorem \ref{thm:3}, inequality \eqref{semicont+} holds for any Lipschitz function $\psi$. Therefore, if $R > 0$ is such that $\O\subseteq B(0,R)$ and we denote by $\mathcal{L}^d$ the Lebesgue measure on $\R^d$, then by setting
\begin{align*}
\psi(x)=\begin{cases}
|x|^p&\text{ if }x\in B(0,R),\\
R^p&\text{ otherwise,}
\end{cases}
\end{align*}
and
$$\mu_n(t)=\frac{1}{\mathcal{L}^d(\O)}\mathcal{L}^d,\qquad\text{for every $n\in\N$ and }t\in[0,T],$$
the lower semicontinuity of \eqref{eq:normp} is a direct consequence of \eqref{semicont+}.
\end{proof}

\section{Concluding remarks}\label{conclusions}

In this paper we addressed the well-posedness of several optimal control problems with Vlasov-type PDE constraints. We first highlighted several crucial features of such PDEs, like the uniform compactness of the support of the trajectories and their smoothness, properties which were then exploited to show the existence of solutions to Problem \ref{prob1}. Several applications of this result were shown by a list of cost functionals falling into our framework, which eventually led us to establish the well-posedness of an evacuation problem encountered in pedestrian dynamics.

A future research direction would be to try to weaken the regularity of the PDE constraints in order to see if our strategy still works. In particular, it could be of interest to try to weaken the regularity assumptions of the class of admissible kernels $\Adm_{\ell}$, allowing for the possibility of nonsmoothness in space given by singularities, see for instance \cite{difrancesco}. It is indeed clear that the closedness properties of our PDE constraints are likely way stronger than necessary. In fact, while the dynamics underlying the control problem is closed under uniform convergence of the trajectories with respect to the Wasserstein distance (see Lemma \ref{le:eqclosed}), what is needed in the proof of Theorem \ref{th:mainmin} is simply the pointwise weak* convergence. 

Concerning the possibility of enlarging the list of functionals presented in Section \ref{exj}, an interesting cost functional that was not included in our analysis is given by
$$\mathcal{J}(\nu,\rho) = \int^T_0 \left|\frac{\partial}{\partial t}\int_{\O}d\nu(t,x)\right|\,dt.$$
This functional penalizes the change of the mass of $\nu$ in time, and appears in contexts where hiring control agents after the dynamics has started is costlier than doing it before. Another functional of interest is
$$\mathcal{J}(\nu,\rho) = \int^T_0 |\nu'(t)|\,dt,$$
where $|\nu'(t)|$ stands for the $d$-metric derivative of $\nu$ at time $t$. The relationship between the above functional and the $\ell_p$--cost \eqref{eq:normp} whenever $\nu$ is subjected to a continuity equation like in \eqref{eq:conteq} is still unclear.

The development of numerical methods for multi-population optimal control problems is a topic that originated a large literature in the last years. Besides the well-established methodology of the discretization of PDE constrained optimal control problems by means of finite element methods, mainly applied for elliptic and parabolic type of equations (see for instance \cite{vexleradaptive}), a particularly promising approach is based on their kinetic description using Boltzmann models, see \cite{albikinetic}. In \cite{albi2015invisible}, the implementation of such methods to solve a control problem similar to Problem \ref{prob2} successfully produced nontrivial optimal strategies, one of which was shown in Figure \ref{fig:meso}. It would be of interest to address in future works the feasibility of these numerical methods for different cost functionals, like those appearing in Section \ref{exj}.

\section*{Acknowledgements}

Mattia Bongini acknowledges the support of the ERC-Starting Grant HDSPCONTR ``High-Dimensional Sparse Optimal Control'' and the fruitful and stimulating discussions with Massimo Fornasier and Francesco Solombrino. The work of Giuseppe Buttazzo is part of the Project 2010A2TFX2 ``Calcolo delle Variazioni'' funded by the Italian Ministry of Research and University. The second author is member of the Gruppo Nazionale per l'Analisi Matematica, la Probabilit\`a e le loro Applicazioni (GNAMPA) of the Istituto Nazionale di Alta Matematica (INdAM).


\appendix

\section{Well-posedness and regularity estimates for \eqref{eq:compacteq}}

The existence of solutions of system \eqref{eq:compacteq} is deeply interwined with that of its discretized counterpart
\begin{equation}\label{discreteeq}
\begin{cases}
\ds\dot{x}_i(t)=\frac1N\sum^N_{j=1}K\big(t,x_i(t)-x_j(t)\big)+f\big(t,x_i(t)\big)\\
x_i(0)=x^N_{0,i}
\end{cases}
\qquad i=1,\ldots,N,
\end{equation}
as the following preliminary result shows.

\begin{proposition}\label{p-rewritten}
	Fix $N \in \N$ and $T>0$. 
	Let $(x^N_{1},\ldots,x^N_{N}):[0,T]\rightarrow\R^{dN}$ be a solution of system \eqref{discreteeq} with initial datum $(x^N_{0,1},\ldots,x^N_{0,N}) \in \R^{dN}$. Then the empirical measure-valued curve $\rho^N:[0,T]\rightarrow\PP(\R^d)$ defined as
	\begin{align*}
	\rho^N(t) = \frac{1}{N}\sum^N_{i = 1}\delta_{x^N_i(t)}, \quad \text{for all } t \in [0,T],
	\end{align*}
	is a solution of \eqref{eq:compacteq} with initial datum
	\begin{align*}
	\rho^N_{0}= \frac{1}{N}\sum^N_{i = 1}\delta_{x^N_{0,i}}.
	\end{align*}
\end{proposition}
\begin{proof}
	Notice that, for all $t \in [0,T]$ and for all $\phi \in \mathcal{C}^{\infty}_0([0,T];\mathcal{C}^{\infty}_b(\O))$, it holds
	\begin{align*}
	\frac{d}{dt}\langle \phi(t,\cdot), \rho^N(t) \rangle &= \frac{1}{N} \sum^N_{i = 1} \frac{d}{dt}\phi(t,x^N_i(t)) \\
	&= \frac{1}{N} \sum^N_{i = 1} \left(\frac{\partial \phi}{\partial t}(t,x^N_i(t)) + \dot{x}_i^N(t) \cdot \nabla\phi(t,x^N_i(t)) \right),
	\end{align*}
	where $\langle \cdot, \cdot \rangle$ denotes the duality pairing between measures and continuous functions.
	By directly applying the expression of $\dot{x}_i^N(t)$ in \eqref{discreteeq} and integrating between $0$ and $t$, we obtain
	\begin{align*}
	\langle \phi(t,\cdot), \rho^N(t)\rangle& - \langle \phi(0,\cdot), \rho^N_0\rangle =\int^t_0 \frac{d}{ds} \langle \phi(s,\cdot), \rho^N(s) \rangle ds \\
	&= \int^t_0\! \int_{\R^d} \!\!\left(\frac{\partial \phi}{\partial s}(s,x) \!+\! \left((K*\rho^N)(s,x) \!+\! f(s,x)\right) \! \cdot \!\nabla\phi(s,x) \!\right) \!d\rho^N(s,x) ds.
	\end{align*}
	Since by assumption $\phi(T,\cdot) = \phi(0,\cdot) \equiv 0$, this shows that $\rho^N$ is a solution of \eqref{eq:compacteq} with initial datum $\rho^N_0$.
\end{proof}

The following result shows that whenever $K$ and $f$ are admissible, solutions of system \eqref{discreteeq} exist and are unique. Its proof is standard, but we report it to show the independence of the result with respect to the discretization parameter $N\in\N$, which plays a crucial role in Theorem \ref{th:bound}.

\begin{lemma}\label{le:discrunif}
	Fix $T>0$, $\ell\in L^1(0,T)$, and $K,f\in\Adm_{\ell}([0,T]\times\R^d;\R^d)$. Suppose that, for every $N\in\N$, $x^N_{0,1},\ldots,x^N_{0,N}\in B$, for some bounded set $B$. Then the system \eqref{discreteeq} has a unique absolutely continuous solution $(x^N_1,\ldots,x^N_N):[0,T]\to\R^{dN}$ in the Carath\'{e}odory sense, see \cite{filippov}. Moreover, there exist $R,L>0$ depending only on $T,B,\ell$ (and thus independent of $N$) such that, for every $N\in\N$ and $i=1,\ldots,N$, it holds
	\begin{itemize}
		\item $|x^N_i(t)|\le R$ for every $t\in[0,T]$;
		\item $|x^N_i(t)-x^N_j(s)|\le\ds L\int^t_s\ell(\theta)\,d\theta$ for every $t,s\in[0,T]$.
	\end{itemize}
\end{lemma}

\begin{proof}
	We begin the proof by showing that, if $K,f\in\Adm_{\ell}([0,T]\times\R^d;\R^d)$ then for any $N\in\N$ the function
	\begin{align*}
	F_N(t,x_1,\ldots,x_N)\!=\!\Bigg(\!\!\frac1N\sum^N_{j=1}K(t,x_1\!-\!x_j)\!+\!f(t,x_1),\ldots,\!\frac1N\sum^N_{j=1}K(t,x_N\!-\!x_j)\!+\!f(t,x_N)\!\Bigg)
	\end{align*}
	belongs to $\Adm_{\gamma}([0,T]\times\R^{dN};\R^{dN})$ for some $\gamma \in L^1(0,T)$. Indeed, the function $F_N$ is Carath\'eodory by definition. Now, fix $x=(x_1,\ldots,x_N)$ and $y=(y_1,\ldots,y_N)\in\R^{dN}$. It holds
	\[\begin{split}
	|F_N(t,x)-F_N(t,y)|
	&\le\sum^N_{i=1}\!\left(\!\frac1N\sum^N_{j=1}\big|K(t,x_i\!-\!x_j)\!-\!K(t,y_i\!-\!y_j)\big|\!+\!\big|f(t,x_i)\!-\!f(t,y_i)\big|\right)\\
	&\le\frac{\ell(t)}{N}\sum^N_{i=1}\sum^N_{j=1}\big|x_i-x_j-y_i+y_j\big|+\ell(t)\sum^N_{i=1}\big|x_i-y_i\big|\\
	&\le3\ell(t)\sum^N_{i=1}\big|x_i-y_i\big|\\
	&\le3\sqrt{N}\ell(t)|x-y|,
	\end{split}\]
	where in the first inequality we estimated the $\ell_2$-norm from above by the $\ell_1$-norm, and in the last one we estimated the $\ell_1$-norm from above by the $\ell_2$-norm times $\sqrt{N}$. A similar computation shows that
	$$|F_N(t,x)|\le3N\ell(t)(1+|x|).$$
	Therefore $F_N\in\Adm_{\gamma}([0,T]\times\R^{dN};\R^{dN})$ for $\gamma = 3N\ell \in L^1(0,T)$, and a usual Cauchy-Lipschitz argument let us conclude that, for every $N\in\N$, system \eqref{discreteeq} has a unique Carath\'eodory solution $(x^N_1,\ldots,x^N_N):[0,T]\to\R^{dN}$.
	
	Let us now fix $N\in\N$ and estimate the growth of $|x_i^N(t)|$ for $i=1,\ldots,N$. Integrating $x_i^N(t)$ in time and taking the norm we obtain
	\begin{align*}
	|x_i^N(t)|&\le|x_{0,i}^N| + \int^t_0 |\dot{x}_i^N(s)| \, ds\\
	&\le|x_{0,i}^N| + \int^t_0 \left(\frac1N\sum^N_{j=1}\big|K\big(t,x^N_i(t)-x^N_j(t)\big)\big|+\big|f\big(t,x^N_i(t)\big)\big|\right) \, ds\\
	&\le|x_{0,i}^N| + \int^t_0 \ell(s)\left(2+2|x_i^N(s)|+\frac1N\sum^N_{j=1}|x_j^N(s)|\right) \, ds.
	\end{align*}
	Set $q^N(t):=\max_{j=1,\ldots,N}|x_j^N(t)|$. Then, the inequalities above imply
	$$q^N(t)\le q^N(0) + \int^t_0 \ell(s)\big(2+3q^N(s)\big) \, ds.$$
	Since $x^N_{0,i}\in B$, 
	from Gronwall's lemma we obtain
	\begin{equation*}
	q_N(t)\le \Bigg(\delta(B) + 2\int^t_0 \ell(s) \, ds\Bigg)\exp\Big(3\int_0^t\ell(s)\,ds\Big),
	\end{equation*}
	where $\delta(B)$ is defined as in \eqref{eq:deltaB}. Therefore, $|x^N_j(t)|\le R$ where
	\begin{equation}\label{Rest}
	R=\Bigg(\delta(B) + 2\int^T_0 \ell(s) \, ds\Bigg)\exp\Big(3\int_0^T\ell(s)\,ds\Big).
	\end{equation}
	This implies that, for all $N\in\N$ and $i= 1,\ldots,N$, we have
	\[\begin{split}
	|\dot{x}^N_i(t)|&\le\frac1N\sum^N_{j=1}\big|K\big(t,x^N_i(t)-x^N_j(t)\big)\big|+\big|f\big(t,x^N_i(t)\big)\big|\\
	&\le\ell(t)\left(2+2|x_i^N(t)|+\frac1N\sum^N_{j=1}|x_j^N(t)|\right)\\
	&\le(2+3R)\ell(t),
	\end{split}\]
	which, integrating between $s$ and $t$ implies
	$$|x_i^N(t)-x_i^N(s)|\le(2+3R)\int^t_s\ell(\theta)\,d\theta.$$
	Setting $L:=2+3R$, the above inequality gives us the uniform continuity of $x^N$ with modulus of continuity uniform in $N$ given by
	\begin{equation}\label{eq:unifmodulus}
	\omega(t,s)=L\int^t_s\ell(\theta)\,d\theta,
	\end{equation}
	which concludes the proof.
\end{proof}

We now establish the existence and uniqueness of solutions of system \eqref{eq:compacteq}. Informally, to do so we consider the solutions $(x^N_1,\ldots,x^N_N)$ of the discrete convolution-type ODE systems \eqref{discreteeq}, write them in the form of empirical measures
\begin{align*}
\rho^N(t) = \frac{1}{N}\sum^N_{i = 1}\delta_{x^N_i(t)}, \quad \text{for all } t \in [0,T],
\end{align*}
and finally take the limit as $N\to\infty$ in the Wasserstein space of probabilities. This procedure, also known as \textit{mean-field limit}, allows us to extend the results obtained in Lemma \ref{le:discrunif} to solutions of \eqref{eq:compacteq}.

We first need a preliminary estimate, a variant of which is Lemma 4.7 of \cite{CanCarRos10}.

\begin{lemma}\label{p-lipkernel}
	Fix $T>0$ and $K\in\Adm_{\ell}([0,T]\times\R^d;\R^d)$, and let $\mu_1,\mu_2:[0,T]\to\mathcal{P}_c(\R^d)$ be two continuous maps with respect to $\W_1$ satisfying for some $R>0$
	\begin{align}\label{eq:bsupp}
	\supp(\mu_i(t))\subseteq B(0,R),\qquad\text{for every }t\in[0,T],\ i=1,2.
	\end{align}
	Then
	$$\|(K*\mu_1)(t,\cdot)-(K*\mu_2)(t,\cdot)\|_{L^{\infty}(\R^d)}\le\ell(t)\W_1(\mu_1(t),\mu_2(t))\qquad\text{for every }t\in[0,T].$$
\end{lemma}

\begin{proof}
	Fix $t\in[0,T]$ and take $\pi\in\Gamma_o(\mu_1(t),\mu_2(t))$. Since the marginals of $\pi$ are by definition $\mu_1(t)$ and $\mu_2(t)$, it follows
	\begin{align*}
	(K*\mu_1)(t,x)-&(K*\mu_2)(t,x)\\
	&=\int_{B(0,R)}\!\!\!K(t,x-y)\,d\mu_1(t,y)-\int_{B(0,R)}\!\!\!K(t,x-z)\,d\mu_2(t,z)\\
	&=\int_{B(0,R)^2}\big(K(t,x-y)-K(t,x-z)\big)\,d\pi(y,z).
	\end{align*}
	By hypothesis \eqref{eq:bsupp} and the $\ell$-admissibility of $K$, we have
	\begin{align*}
	\big|(K*\mu_1)(t,x)-(K*\mu_2)(t,x)\big|&\le\int_{B(0,R)^2}\big|K(t,x-y)-K(t,x-z)\big|\,d\pi(y,z)\\
	&\le\ell(t)\int_{B(0,R)^2}|y-z|\,d\pi(y,z)\\
	&=\ell(t)\W_1(\mu_1(t),\mu_2(t)),
	\end{align*}
	which concludes the proof.
\end{proof}

We are finally ready to prove Theorem \ref{th:bound}.

\begin{proof}[Proof of Theorem \ref{th:bound}]
	For every $N\in\N$, let $x^N_{0,1},\ldots,x^N_{0,N}$ be such that the empirical measure
	\begin{align*}
	\rho^N_0=\frac{1}{N}\sum^N_{i=1}\delta_{x^N_{0,i}}
	\end{align*}
	tends to $\rho_0$ weakly*, hence $\W_1(\rho_0,\rho^N_0)\to0$ as $N\to\infty$. For every $N\in \N$, consider now the unique solution $x^N = (x^N_1,\ldots,x^N_N)$ of system \eqref{discreteeq} with initial datum $(x^N_{0,1},\ldots,x^N_{0,N})$, and denote by
	\begin{align*}
	\rho^N(t) = \frac{1}{N}\sum^N_{i = 1} \delta_{x^N_{i}(t)}, \quad \text{ for every } t\in [0,T],
	\end{align*}
	the empirical measure curve supported on the trajectories of $x^N$. From Proposition \ref{p-rewritten} follows that $\rho^N$ is the solution of \eqref{eq:compacteq} with initial datum $\rho^N_0$.
	
	By Lemma \ref{le:discrunif}, the elements of the sequence $(\rho^N)_{N \in \N} \subset \mathcal{C}([0,T];\mathcal{P}_1(B(0,R)))$ have support uniformly contained in the ball $B(0,R)$, where $R$ is given by \eqref{Rest}, and they are uniformly continuous with modulus of continuity $\omega$ given by \eqref{eq:unifmodulus} uniform in $N$. 
	
	Hence the following holds:
	\begin{itemize}
		\item $(\rho^N)_{N \in \N}$ is equicontinuous and is contained in a closed subset of $\mathcal{C}([0,T];\mathcal{P}_1(B(0,R)))$, because of the uniform modulus of continuity;
		\item for every $t \in [0,T]$, the sequence $(\rho^N(t))_{N \in \N}$ is relatively compact in $\mathcal{P}_1(B(0,R))$ equipped with the $\W_1$ metric. This holds because $(\rho^N(t))_{N \in \N}$ is a tight sequence, since $B(0,R)$ is compact, and hence relatively compact with respect to the weak* convergence due to Prokhorov's Theorem. By Proposition 7.1.5 of \cite{AGS} and the uniform integrability of the first moments of the family $(\rho^N(t))_{N \in \N}$ follows the relative compactness also in the metric space $(\mathcal{P}_1(B(0,R)),\W_1)$.
	\end{itemize}
	Therefore, we can apply the Ascoli-Arzel\'{a} Theorem for functions with values in a metric space 
	to infer the existence of a subsequence $(\rho^{N_k})_{k \in \N}$ of $(\rho^N)_{N \in \N}$ such that
	$$\lim_{k\to\infty}\W_1(\rho^{N_k}(t),\rho(t)) = 0 \quad \text{ uniformly for a.e. } t \in [0,T],$$
	for some uniformly continuous curve $\rho \in \mathcal{C}([0,T];\mathcal{P}_1(B(0,R)))$, again with $\omega$ as modulus of continuity. The property that $\W_1(\rho^N_0,\rho_0) \rightarrow 0$ as $N\to\infty$ now obviously implies $\rho(0) = \rho_0$.
	
	We are now left with verifying that $\rho$ is a solution of \eqref{eq:compacteq}. From the computations in Proposition \ref{p-rewritten} follows that
	for all $t \in [0,T]$ and for all $\phi \in \mathcal{C}^{\infty}_0([0,T];\mathcal{C}^{\infty}_b(\O))$ it holds
	\begin{align*}
	\langle \phi(t,\cdot), \rho^N(t)\rangle& - \langle \phi(0,\cdot), \rho^N_0\rangle =\int^t_0 \frac{d}{ds} \langle \phi(s,\cdot), \rho^N(s) \rangle ds \\
	&= \int^t_0\! \int_{\R^d} \!\!\left(\frac{\partial \phi}{\partial s}(s,x) \!+\! \left((K*\rho^N)(s,x) \!+\! f(s,x)\right) \! \cdot \!\nabla\phi(s,x) \!\right) \!d\rho^N(s,x) ds.
	\end{align*}
	We now want to prove that
	\begin{align*}
	\lim_{N\to\infty} \int^t_0 \int_{\R^d} &\left(\frac{\partial \phi}{\partial s}(s,x) + \left((K*\rho^N)(s,x) + f(s,x)\right) \cdot \nabla\phi(s,x) \right) d\rho^N(s,x) ds \\
	&=\int^t_0 \int_{\R^d} \left(\frac{\partial \phi}{\partial s}(s,x) + \left((K*\rho)(s,x) + f(s,x)\right) \cdot \nabla\phi(s,x) \right) d\rho(s,x) ds.
	\end{align*}
	To do so, notice that by Lemma \ref{p-lipkernel} and the uniform $\W_1$ convergence of the $\rho^N$ to $\rho$, it holds
	\begin{align*}
	\int^t_0 \int_{\R^d} &\left|\left((K*\rho^N)(s,x) - (K*\rho)(s,x) \right) \cdot \nabla\phi(s,x) \right| d\rho(s,x) ds \\
	&\leq \int^t_0 \ell(s) \W_1(\rho^N(s),\rho(s))\left[\int_{\R^d} |\nabla\phi(s,x) | d\rho(s,x)\right] ds \\
	& \leq \sup_{t \in [0,T]} \W_1(\rho^N(t),\rho(t)) \int^t_0 \ell(s) \left[\int_{\R^d} |\nabla\phi(s,x) | d\rho(s,x)\right] ds \\
	&\to0\text{ as }N\to\infty,
	\end{align*}
	since $\ell \in L^1(0,T)$, $\nabla \phi$ is bounded and $\rho$ has compact support.
	
	
	Therefore, since by assumption $\phi(T,\cdot) = \phi(0,\cdot) \equiv 0$, we obtain from the dominated convergence theorem
	\begin{align*}
	\int^T_0 \int_{\R^d} \left(\frac{\partial \phi}{\partial s}(s,x) + \left((K*\rho)(s,x) + f(s,x)\right) \cdot \nabla\phi(s,x) \right) d\rho(s,x) ds = 0,
	\end{align*}
	which proves that $\rho$ is a solution of \eqref{eq:compacteq} with initial datum $\rho_0$.
	
	The uniqueness of $\rho$ is a consequence of Theorem 3.10 of \cite{CanCarRos10}.
\end{proof}


\end{document}